\documentclass[12pt]{article}
\usepackage{amsmath,amssymb,amsthm,latexsym}
\usepackage{graphicx}
\usepackage{xcolor}
\usepackage{hyperref}
\usepackage{geometry}
\numberwithin{equation}{section}
\newtheorem{theorem}{Theorem}[section]
\newtheorem{lemma}{Lemma}[section]

\newtheorem{proposition}{Proposition}[section]

\newtheorem{remark}{Remark}[section]

\newtheorem{thm}{Theorem}
\title{\textbf{On the first eigenvalue of the area Jacobi operator for complex curves in K\"ahler surfaces}}
\author {Zhenxiao Xie\\}
\date{}
\begin{document}
\maketitle
\begin{abstract}
In this paper, we investigate the first eigenvalue $\Lambda_1$ of the area Jacobi operator for complex curves in K\"ahler surfaces, establishing an extrinsic counterpart to the classical Lichnerowicz theorem for the Laplace-Beltrami operator. By analyzing the second variation of a conformally invariant Willmore-type functional, we derive the lower bound  $\Lambda_1 \geq 2\,\mathfrak{Ric}$, where $\mathfrak{Ric}$ denotes the infimum of the ambient Ricci curvature. For K\"ahler-Einstein surfaces with positive Einstein constant $\mathfrak{c}>0$, this bound reduces to $\Lambda_1 \geq 2\mathfrak{c}$. We then explore the equality case, computing the exact dimension of the corresponding first eigenspace in terms of the area, genus, and the dimension of a space of holomorphic sections. This analysis shows that the equality is achieved for all curves of genus $g \leq 1$. 
\end{abstract}

\indent{\bf Keywords:} first eigenvalue; area Jacobi operator; complex curves; Willmore functional 
\vskip 0.2cm

\indent{\bf MSC(2020):\hspace{2mm} 53A10 53C55 58J50} 

\section{Introduction}
The first eigenvalue of the Laplace-Beltrami operator on a compact Riemannian manifold encodes crucial geometric information, including curvature bounds. A fundamental result in this direction is the Lichnerowicz theorem~\cite{Lichnerowicz}, which states that if the Ricci curvature of an \(n\)-dimensional compact manifold satisfies \(\mathrm{Ric} \ge (n-1)K\) for some \(K>0\), then the first positive eigenvalue \(\lambda_1\) obeys \(\lambda_1 \ge nK\). This estimate is sharp, and the Obata theorem~\cite{Obata1962} characterizes the equality case: \(\lambda_1 = nK\) holds iff the manifold is isometric to a sphere of constant sectional curvature \(K\). In the K\"ahler setting, Futaki~\cite{Futaki} proved that the first eigenvalue of the complex Laplacian is bounded below by \(k\) for a compact K\"ahler manifold satisfying \(\mathrm{Ric} \ge k\) with \(k>0\). The equality case was later characterized by Tam and Yu~\cite{Tam-Yu}.  

The area Jacobi operator $\mathcal{L}$, which governs the second variation of the area functional for minimal submanifolds, is a second-order, self-adjoint elliptic operator. As such, it can be regarded as the extrinsic counterpart to the intrinsic Laplace–Beltrami operator. Unlike the latter, it is not positive (or even nonnegative) in general. We denote by $\Lambda_1$ the smallest nonzero eigenvalue of $\mathcal{L}$, referred to as the first eigenvalue. Since $\mathcal{L}$ serves as the stability operator, $\Lambda_1$ governs the infinitesimal behavior of the minimal submanifold under deformations. A positive $\Lambda_1$ means the submanifold is locally minimizing, whereas a negative one indicates instability.
In recent decades, considerable attention has been devoted to estimating the first eigenvalue of the area Jacobi operator for minimal surfaces in various ambient spaces. 
In \cite{Perdomo}, Perdomo proved that if $\Sigma^2$ is a compact oriented minimal surface which is not totally geodesic in $\mathbb{S}^3$, then the first eigenvalue 
of its area Jacobi operator satisfies $\Lambda_1 \leq -4$, with equality if and only if $\Sigma^2$ is isometric to the Clifford torus.  
This result sparked a series of further estimates for $\Lambda_1$ on closed hypersurfaces of constant mean curvature or conatnt scalar curvature (see \cite{Alias,Ambrozio,BatistaCavalcanteMelo,Batista,Bray,Chen-Cheng,Cheng,Dung,Li-Wang,Merono,Zhu} and references therein). 

Micallef and Wolfson~\cite{MW} obtained fundamental results on the stability of minimal surfaces in 4-dimensional manifolds by considering a complex version of the second variation of area. Subsequently, for superminimal surfaces with negative spin in self-dual Einstein $4$-dimensional manifolds, Montiel and Urbano~\cite{MU} 
proved that the first eigenvalue of the area Jacobi operator satisfies 
$\Lambda_1\geq-2\mathfrak{c}$, where $\mathfrak{c}$  denotes the Einstein constant of the ambient space, and they computed exactly the index 
explicitly in terms of the genus and area of the surface. 

This paper focuses on another important class of codimension-2 minimal surfaces, i.e., complex curves in Kähler surfaces, which are superminimal surfaces of positive spin. It is a classical fact that such surfaces are absolutely area‑minimizing and thus stable. Consequently, their associated area Jacobi operator $\mathcal{L}$ is a non‑negative self‑adjoint elliptic operator, exactly analogous to the Laplace–Beltrami operator on a Riemannian manifold. While the non-negativity of $\mathcal{L}$ is well understood, obtaining sharp quantitative estimates for its spectrum in terms of the ambient curvature is a more subtle problem, and the focus of our investigation. 
We established a universal lower bound for the first eigenvalue $\Lambda_1$ of $\mathcal{L}$. To state our result, we 
denote by 
\[
\mathfrak{Ric} \triangleq \inf_{\substack{p \in \Sigma, \\ u \in UT_p M^4}} \operatorname{Ric}(u, u),
\]
the infimum of the Ricci curvature over all unit tangent vectors at all points of $M^4$. 

\begin{thm}\label{thm-main11}
    For any complex curve in a K\"ahler surface $M^4$, the first eigenvalue $\Lambda_1$ of its area Jacobi operator $\mathcal{L}$ satisfies
    $$\Lambda_1\geq 2\,\mathfrak{Ric}.$$
\end{thm}
Somewhat surprisingly, the proof of this theorem relies on computing the second variation of the following  conformally invariant functional 
$$\mathcal{W}^+(x) \triangleq \int_{\Sigma} (|\vec H|^2 + K(\mathfrak{e}_1, \mathfrak{e}_2, \mathfrak{e_1}, \mathfrak{e_2}) + K(\mathfrak{e}_1, \mathfrak{e}_2, \mathfrak{e_3}, \mathfrak{e_4}))dA,$$
introduced by Montiel and Urbano in \cite{Montiel-Urbano} for closed oriented surfaces in an oriented $4$-dimensional Riemannian manifold. Here $\vec H$ represents the mean curvature vector, $K$ denotes the ambient  curvature tensor, and $\{\mathfrak{e}_1, \mathfrak{e}_2, \mathfrak{e}_3, \mathfrak{e}_4\}$ is an orthonormal basis with $\{\mathfrak{e}_1, \mathfrak{e}_2\}$ tangent to the surface $x$ and $\{\mathfrak{e}_3, \mathfrak{e}_4\}$ spanning its normal space. 
Associated with $\mathcal{W}^+$, Montiel and Urbano \cite{Montiel-Urbano} also introduced  
$$\mathcal{W}^-(x) \triangleq \int_{\Sigma} (|\vec H|^2 + K(\mathfrak{e}_1, \mathfrak{e}_2, \mathfrak{e_1}, \mathfrak{e_2}) - K(\mathfrak{e}_1, \mathfrak{e}_2, \mathfrak{e_3}, \mathfrak{e_4}))dA.$$
The average of these two functionals is precisely the Willmore functional,
$$\mathcal{W}(x) = \int_{\Sigma} (|\vec H|^2 + K_{1212}) dA,$$
an important global conformal invariant for surfaces in general curved Riemannian manifolds that has attracted considerable attention recently (see \cite{Marques1,Mondino,Mondino1,Wang-Xie} and the references therein). In \cite{Montiel-Urbano}, Montiel and Urbano studied $\mathcal{W}^+$ and $\mathcal{W}^-$ from the perspective of twistor theory and showed that a superminimal surface with positive (resp. negative) spin minimize $\mathcal{W}^+$ (resp. $\mathcal{W}^-$) among its homotopy class. In \cite{Wang-Xie}, we investigated these functionals from the viewpoint of conformal geometry and derived the following expressions, 
\begin{equation}\label{eq-W+}
\mathcal{W}^+(x) = 4 \int_{\Sigma} |\psi|^2 dA + 2\pi (\chi + \chi^{\perp}),~~~\mathcal{W}^-(x) = 4 \int_{\Sigma} |\phi|^2 dA + 2\pi (\chi - \chi^{\perp}),
\end{equation}
where 
$\psi$ and $\phi$ are two local conformal invariants arising from the normal-bundle-valued Hopf differential, and $\chi$ (resp. $\chi^{\perp}$) denotes the Euler characteristics of the tangent bundle (resp. normal bundle). When the ambient space is a Kähler surface, complex curves are superminimal surfaces with positive spin and hence are minimizers of $\mathcal{W}^+$.   

If we restrict the ambient space to a Kähler–Einstein surface with positive scalar curvature, then the estimate established in Theorem \ref{thm-main11} is sharp. In this setting, the curvature invariant $\mathfrak{Ric}$ simplifies to the Einstein constant $\mathfrak{c}>0$. We explore the equality case and are able to compute exactly the dimension of the first eigenspace. This is accomplished by relating it to the space of Jacobi fields of $\mathcal{W}^+$, which corresponds precisely to the preimage of the space of holomorphic sections of a certain line bundle, and then applying the Riemann-Roch theorem in conjunction with the Dolbeault theorem. 
\begin{thm}\label{thm-main2}
 Let \(M^{4}\) be a K\"ahler-Einstein surface with Einstein constant \(\mathfrak{c}>0\), and let \(x\colon\Sigma \rightarrow M^{4}\) be a complex curve. Then 
 \begin{itemize}
    \item[{\rm(1)}] the first eigenvalue $\Lambda_1$ of its area Jacobi operator is at least $2\mathfrak{c}$,
    \item[{\rm(2)}] if $\Sigma$ has genus $g\leq 1$, then $\Lambda_1=2\mathfrak{c}$ and the corresponding eigenspace has dimension
    \[
    \mathfrak{c}\frac{\operatorname{Area}}{2\pi}+1-g+\dim H^{\,0}(N^*_\Sigma\otimes K_\Sigma^2),
    \]
    where $N^*_\Sigma$ is the dual of the normal bundle, $K_\Sigma$ is the canonical line bundle of $\Sigma$, and $H^{\,0}(N^*_\Sigma\otimes K_\Sigma^2)$ denotes  the space of holomorphic sections of $N^*_\Sigma\otimes K_\Sigma^2$. 
\end{itemize}
\end{thm}

This paper is organized as follows. In Section~\ref{sec2}, we review the necessary background on surface theory in  Kähler surfaces. Section~\ref{sec3} is devoted to compute the variation of geometric quantities for a complex curve in a K\"ahler surface. Theorem~\ref{thm-main11} is established in Section~\ref{sec4}. Finally, in Section~\ref{sec5}, we focus on the Kähler-Einstein setting and prove Theorem \ref{thm-main2}.

\section{Preliminaries}\label{sec2}
Let $ M^4 $ be a Kähler surface, and $ x: \Sigma \to M^4 $ a real surface in $ M^4 $. As in \cite{Chern-Wolfson}, we can choose, along $x$, a unitary coframe $ \{\omega_0, \omega_1\}$, with the corresponding dual frame $\{e_0, e_1\}$, such that
\begin{equation}\label{eq-pull}
x^* \omega_1 = \lambda \theta, \quad x^* \omega_0 = \mu \bar{\theta},
\end{equation}
where $ \theta $ is a $ (1,0) $ form on $ \Sigma $ such that the induced metric 
$$ dx \cdot dx =\frac{\theta\otimes \bar\theta+\bar\theta\otimes \theta}{2}={|\theta|^2},$$ 
and $ \lambda, \mu $ are two local complex functions such that $ |\lambda|^2 + |\mu|^2 = 1 $. Note that $x$ is a complex curve if and only if either $\lambda$ or $\mu$ vanishes identically. We denote by $ \omega_{j \bar{k}} $ the connection forms of $ M^4 $ given by 
\begin{equation}\label{eq-con}
 d\omega_j=\omega_{j\bar k}\wedge \omega_{k}, ~~~0\leq j,k\leq 1.   
\end{equation}
They satisfy 
\begin{equation}\label{eq-conn}
\omega_{j\bar k}+\overline{\omega}_{k\bar j}=0, ~~~\nabla e_j=-\omega_{j\bar k} e_k,~~~0\leq j,k\leq 1,\end{equation}
where $\nabla$ denotes the Levi-Civita connection of $M^4$. The curvature form $\Omega_{i\bar j}$ is defined by 
$$d\omega_{i\bar j}-\omega_{i\bar k}\wedge \omega_{k\bar j}=\Omega_{i\bar j}.$$

We denote by $N_\Sigma$ the normal bundle of $x$, and $N_{\Sigma}^{\mathbb{C}}\triangleq N_\Sigma\otimes\mathbb{C}$ its complexification. It follows from \eqref{eq-pull} that 
\begin{equation}\label{eq-normal}
\bar{\mu} \, x^* \omega_1 - \lambda \, x^* \bar{\omega}_0 = 0, 
\end{equation}
which implies $$ n \triangleq \bar{\lambda} \bar{e}_0-\mu e_1 $$ defines a local section of $N_{\Sigma}^{\mathbb{C}}$. Taking the differential of \eqref{eq-normal} yields 
\begin{equation}\label{eq-sfm1}
(\lambda \, d \bar{\mu} - \bar{\mu} \, d \lambda) + \lambda \bar{\mu} \, (x^* \omega_{1 \bar{1}} + x^* \omega_{0 \bar{0}}) = a \theta + b \bar{\theta}, 
\end{equation}
\begin{equation}\label{eq-sfm2}
x^* \omega_{1 \bar{0}} = b \theta + c \bar{\theta}, 
\end{equation}
where $ a $, $ b $ and $ c $ are local complex-valued functions that characterize the second fundamental form of $ x $. In fact, in terms of ${a, b, c}$, the mean curvature vector of $x$ takes the form 
\begin{equation}\label{eq-vH}
\vec{H} = 2b \, n + 2\bar{b} \, \bar{n},
\end{equation}
while the normal-bundle-valued Hopf differential is given by 
$$
2({a} {n} + \bar{c} \bar{n}) \theta^2.
$$
In terms of the notation introduced in \cite{Wang-Xie}, we have $\psi=a$ and $\phi=\bar c$. Thus, the conformally invariant functional $\mathcal{W}^+$ (see \eqref{eq-W+} in the introduction) now takes the form 
\begin{equation}\label{eq-W+1}
\mathcal{W}^+(x) = 4 \int_{\Sigma} |a|^2 dA + 2\pi (\chi + \chi^{\perp}).
\end{equation}
If $ a \equiv 0 $ (resp. $ c \equiv 0 $), then $ x $ is called a super-conformal surface with positive (resp. negative) spin. If, in addition, $ b \equiv 0 $, then $ x $ is referred to as a \emph{superminimal} surface with positive (resp. negative) spin. 

Now we assume that 
$x$ is a complex curve, then \eqref{eq-sfm1} yields 
$$ a \equiv 0,~~~ b \equiv 0,$$
which implies $ x $ is a superminimal surface with positive spin. Consequently, $x$ minimizes the functional $\mathcal{W}^+$ within its homotopy class. Without loss of generality, we assume $\lambda\equiv 1$ and $\mu\equiv 0$. Denote by $\theta_{12}$ (resp. $\theta_{34}$) the connection form of the tangent (resp. normal) bundle of $x$, 
It follows from \eqref{eq-pull}$\sim$\eqref{eq-conn} 
that 
\begin{equation}\label{eq-o1100}
 x^*\omega_{1\bar{1}} = -i\theta_{12}, ~~~x^*\omega_{0\bar{0}} = -i\theta_{34}.   
\end{equation}
Differentiating \eqref{eq-o1100} yields the Gauss-Ricci equation
\begin{equation}\label{eq-Gauss1}
    (R_{1212}+R_{1234})\,\theta\wedge\bar\theta = -2x^*(\Omega_{0\bar0}+\Omega_{1\bar1}),
\end{equation}
where \(R_{1212}\) (resp. \(R_{1234}\)) denotes the curvature of the tangent (resp. normal) bundle, defined by
\[
d\theta_{12} = \frac{R_{1212}}{2i}\,\theta\wedge\bar\theta, \qquad 
d\theta_{34} = \frac{R_{1234}}{2i}\,\theta\wedge\bar\theta.
\]
Note that 
$-i(\Omega_{0\bar0}+ \Omega_{1\bar1})$ is exactly the Ricci form of $M^4$, which implies  
\begin{equation}\label{eq-Ric}
x^*(\Omega_{0\bar0}+ \Omega_{1\bar1})=-\mathrm{Ric}(e_1, \bar e_1)\theta\wedge\bar\theta.
\end{equation}
Therefore, we can also reformulate the Gauss-Ricci  equation \eqref{eq-Gauss1} as 
\begin{equation}\label{eq-Gauss2}
    R_{1212}+R_{1234}=2\mathrm{Ric}(e_1,\bar e_1).
\end{equation}

\section{The variation of a complex curve in a K\"ahler surface}\label{sec3}
Let $x:\Sigma\rightarrow M^4$ be a real surface in the K\"ahler surface $M^4$. 
Consider the variation of $ x $ given by  
$$
X: \Sigma \times (-\varepsilon, \varepsilon) \longrightarrow M^4.
$$
Set $ x_t \triangleq X(\cdot, t) $. For $ t $ sufficiently small, $ x_t $ also defines an immersion of $ \Sigma $ in $ M^4 $. By abuse of notation, we adopt the same symbols as in the previous chapter. 
Along this variation, we choose a unitary coframe $ \{ \omega_0, \omega_1 \} $ and the dual frame $ \{ e_0, e_1 \} $ such that restricted on each $ x_t $, they are adapted, i.e., there exist two locally defined complex-valued functions $\lambda$ and $\mu$ on $\Sigma \times (-\varepsilon, \varepsilon)$ such that 
\begin{equation}\label{eq-pul}
x_t^* \omega_1 = \lambda \, \theta, \quad x_t^* \omega_0 = \mu \, \overline{\theta},
\end{equation}
where $ \theta $ (may depend on $t$) is a locally defined (1,0)-form on $ \Sigma $ with respect to the induced metric $dx_t \cdot dx_t$ such that $ dx_t \cdot dx_t = |\theta|^2 $.
It follows that 
$$
\overline{\mu} \, x_t^* \omega_1 - \lambda \, x_t^* \overline{\omega_0} = 0,
$$
which allows us to define three local complex-valued functions $a, b, c$ on $\Sigma\times (-\epsilon, \epsilon)$ as in \eqref{eq-sfm1} and \eqref{eq-sfm2}. 
With the parameter $t$ fixed, $n \triangleq  \overline{\lambda} \overline{e}_0-\mu e_1 $ defines a complex-valued normal vector of $ x_t $, and the second fundamental form of $x_t$ is completely determined by $a, b$ and $c$.


In the sequel, we assume that $ x $ is a complex curve and, without loss of generality, that 
\begin{equation*}
    \lambda(\cdot, 0) \equiv 1,~~~  \mu(\cdot, 0) \equiv 0.
\end{equation*}
Let   
$X$ be a normal variation with 
\begin{equation}\label{eq-V}
V \triangleq \left.\frac{\partial X}{\partial t}\right|_{t=0}  = \bar{\nu} n + \nu \bar n=\bar{\nu} \bar e_0 + \nu e_0
\end{equation}
as the variation vector field, where $ \nu $ is a locally defined complex-valued function on $\Sigma$. 
We define the complex covariant differentials of $\nu$ and $\bar{\nu}$ as follows, 
\begin{equation}\label{eq-Dv}
\nu_{1}\theta + {\nu}_{\bar 1}\bar{\theta} = D\nu \triangleq d_\Sigma \nu + i\theta_{34}\nu,~~~~~~\bar{\nu}_{1}\theta + \bar{\nu}_{\bar 1}\bar{\theta} = D\bar\nu \triangleq d_\Sigma \bar\nu - i\theta_{34}\bar \nu,
\end{equation}
so that $\nabla V=D\bar{\nu}\, \bar{e}_0+D{\nu}\, {e}_0$, where $d_\Sigma$ denotes the ordinary exterior differential operator on $\Sigma$. Similarly, the second order complex covariant derivatives of $\nu$ and $\bar\nu$ are defined as 
\begin{equation}\label{eq-v1vb1}
\begin{split}\nu_{11}\theta + {\nu}_{1\bar 1}\bar{\theta} = D\nu_1 \triangleq d_\Sigma \nu_1 - (i\theta_{12} - i\theta_{34})\nu_1,\\\bar\nu_{11}\theta + \bar{\nu}_{1\bar 1}\bar{\theta} = D\bar\nu_1 \triangleq d_\Sigma \bar\nu_1 - (i\theta_{12} + i\theta_{34})\bar\nu_1,\\{\nu}_{\bar 11}\theta + \nu_{\bar{1}\bar{1}}\bar{\theta} = D{\nu}_{\bar{1}} \triangleq d_\Sigma {\nu}_{\bar{1}} + (i\theta_{12} + i\theta_{34}){\nu}_{\bar 1},\\\bar{\nu}_{\bar 11}\theta + \bar\nu_{\bar{1}\bar{1}}\bar{\theta} = D\bar{\nu}_{\bar{1}} \triangleq d_\Sigma \bar{\nu}_{\bar{1}} + (i\theta_{12} - i\theta_{34})\bar{\nu}_{\bar 1}.
\end{split}
\end{equation}
It is obvious that 
\begin{equation}\label{eq-bnu}
\bar{\nu}_{\bar1}=\overline{\nu_1},~~~\bar{\nu}_{1}=\overline{\nu_{\bar1}},~~~\bar\nu_{11}=\overline{\nu_{\bar{1}\bar{1}}},~~~\bar\nu_{\bar{1}\bar{1}}=\overline{\nu_{11}},~~~\bar{\nu}_{\bar 11}=\overline{{\nu}_{1\bar 1}}, ~~~\bar{\nu}_{1\bar 1}=\overline{{\nu}_{\bar 11}}.
\end{equation}

The following lemma is a special case of (75) and (76) in our previous work \cite{Wang-Xie}. For the convenience of the reader and to keep the exposition self-contained, we include a detailed proof here. 
\begin{lemma}For a complex curve $x:\Sigma\rightarrow M^4$, the variations of $a$ and $b$ at $t=0$ are given by  
\begin{equation}\label{eq-abt}
\left.\frac{\partial a}{\partial t}\right|_{t=0} = \bar\nu_{11}, \qquad 
\left.\frac{\partial b}{\partial t}\right|_{t=0} = \bar\nu_{1\bar 1}.  
\end{equation}
\end{lemma}
\begin{proof}
Taking the first order derivative with respect to \(t\) at \(t=0\) of the equation
\[
(\lambda d_{\Sigma} \bar{\mu} - \bar{\mu} d_{\Sigma} \lambda) + \lambda \bar{\mu} (x_t^* w_{1\bar{1}} + x_t^* w_{0\bar 0}) = a \theta + b \bar{\theta}
\]
gives
\begin{align}\label{eq-papa}
\left.\frac{\partial a}{\partial t}\right|_{t=0} \theta + \left.\frac{\partial b}{\partial t}\right|_{t=0} \bar\theta 
&= d_\Sigma \left.\frac{\partial \bar\mu}{\partial t}\right|_{t=0} + \left.\frac{\partial \bar\mu}{\partial t}\right|_{t=0} (x^*\omega_{1\bar1} + x^*\omega_{0\bar 0}). 
\end{align}

Next, we compute $\left.\frac{\partial \bar\mu}{\partial t}\right|_{t=0}$. 
By the expression of the variation vector field $V$ in \eqref{eq-V}, we have
\begin{align}X^*\omega_{0} &= x_t^*\omega_{0} + {\nu}dt = \mu\bar{\theta} + {\nu}dt, \label{eq-o0}\\
X^*\omega_{1} &= x_t^*\omega_{1} = \lambda\theta. \label{eq-o1}
\end{align}
Set 
\begin{equation}\label{eq-o01}
X^*\omega_{0\bar 1} \triangleq x_t^*\omega_{0\bar 1} + A_{0\bar 1}dt,    
\end{equation}
where $A_{0\bar 1}$ is a locally defined complex-valued function. 
Note that, on $\Sigma \times (-\epsilon, \epsilon)$  there holds 
\begin{equation*}\label{eq-d}
d = d_{\Sigma} + dt \wedge \frac{\partial}{\partial t}.
\end{equation*}
Apply this to the differential of both hand sides of \eqref{eq-o0}. At $t=0$, the right hand side equals
$$\begin{aligned}
(d_{\Sigma} + dt \wedge \frac{\partial}{\partial t})(x_t^*\omega_{0} + \nu dt)
= d_{\Sigma}x_t^*\omega_{0} + d_{\Sigma}{\nu} \wedge dt + \left.\frac{\partial \mu}{\partial t}\right|_{t=0} dt \wedge \bar{\theta},
\end{aligned}$$
while the left hand side is given by
\[
\begin{aligned}
X^* d\omega_0
&= X^*(\omega_{0\bar 0} \wedge \omega_0 + \omega_{0\bar 1} \wedge \omega_1)= X^*\omega_{0\bar 0} \wedge X^*\omega_0 + X^*\omega_{0\bar 1} \wedge X^*\omega_1 \\
&=\nu x_t^*\omega_{0\bar 0} \wedge dt + x_t^*\omega_{0\bar 0} \wedge x_t^*\omega_0 + x_t^*\omega_{0\bar 1} \wedge x_t^*\omega_1+ A_{0\bar 1} dt\wedge x_t^*\omega_1 \\
&=\nu x_t^*\omega_{0\bar 0} \wedge dt +  x_t^*d\omega_0 + A_{0\bar 1} dt \wedge \theta \\
&= -i \nu \theta_{34} \wedge dt + \lambda A_{0\bar 1}  dt \wedge \theta +  d_{\Sigma}x_t^* \omega_0,
\end{aligned}
\]
where we have used \eqref{eq-o1100} and \eqref{eq-o0}$\sim$
\eqref{eq-o01}.  
Hence, we obtain
\[
\Bigl( d_{\Sigma} \nu + i\theta_{34}\nu + \lambda A_{0\bar 1}\theta- \left.\frac{\partial \mu}{\partial t}\right|_{t=0} \bar{\theta} \Bigr) \wedge dt = 0,
\]
which yields at \(t=0\),
\begin{equation}\label{eq-nu1}
A_{0\bar 1} = -{\nu}_1, ~~~~~~ \left.\frac{\partial \mu}{\partial t}\right|_{t=0} = {\nu}_{\bar1}.    
\end{equation}
Substituting \eqref{eq-nu1} into \eqref{eq-papa} together with \eqref{eq-nu1} gives 
\[
\begin{aligned}
\left.\frac{\partial a}{\partial t}\right|_{t=0} \theta + \left.\frac{\partial b}{\partial t}\right|_{t=0} \bar\theta 
&= d_\Sigma\,\bar\nu_1 - \bar\nu_1 (i\theta_{12} + i\theta_{34})= \bar\nu_{11} \theta + \bar\nu_{1\bar 1} \bar\theta.
\end{aligned}
\]
This finishes the proof.

\end{proof}
In the sequel, we denote by $\mathcal{A}$ the area functional. 
\begin{proposition}
    For a complex curve $x: \Sigma\rightarrow M^4$, the following second variations hold:   
    \begin{equation}\label{eq-Area}
        \left.\frac{\partial^2}{\partial t^2} \right|_{t=0}\mathcal{A}(x_t)
= 4 \int_{\Sigma} |\nu_{\bar1}|^2 \, dA,
    \end{equation}
        \begin{equation}\label{eq-Will}
        \left.\frac{\partial^2}{\partial t^2} \right|_{t=0}\mathcal{W}^+(x_t)
= 8 \int_{\Sigma} |\nu_{\bar1\bar 1}|^2 \, dA .
    \end{equation}
\end{proposition}
\begin{proof}
It is well known that 
\[
\left.\frac{\partial^2}{\partial t^2} \right|_{t=0}\mathcal{A}(x_t)=\left.\frac{\partial^2}{\partial t^2} \right|_{t=0} \int_{\Sigma} dA_t 
= -2 \int_{\Sigma} \Bigl\langle \left.\frac{\partial \vec{H}}{\partial t} \right|_{t=0},\, {V} \Bigr\rangle \, dA_t .\]
It follows from \eqref{eq-vH}, \eqref{eq-bnu} and  \eqref{eq-abt} that 
\begin{equation}\label{eq-Ht}
\left.\frac{\partial \vec{H}}{\partial t} \right|_{t=0}=2\bar\nu_{1\bar1}n+ 2\nu_{\bar11}\bar n,
\end{equation}
which yields that 
$$\left.\frac{\partial^2}{\partial t^2} \right|_{t=0}\mathcal{A}(x_t)= -2 \int_{\Sigma} (\nu \bar\nu_{1\bar 1}+\bar\nu \nu_{\bar1 1}) \, dA_t.$$
Then \eqref{eq-Area} follows from integration by parts. 

By the definition, 
$$\frac{\partial }{\partial t} \mathcal{W}^+(x_t)=4\frac{\partial }{\partial t} \int_\Sigma |a|^2 dA_t=4\int_\Sigma \left(a \frac{\partial \bar a }{\partial t}+\bar a \frac{\partial a}{\partial t}+2|a|^2 \bigl\langle \vec H,  V\bigr\rangle\right) dA_t.$$
Then using \eqref{eq-abt}, we obtain 
$$\left.\frac{\partial^2}{\partial t^2} \right|_{t=0}\mathcal{W}^+(x_t)=8\int_\Sigma \left|\frac{\partial a }{\partial t}\Big|_{t=0} \right|^2dA
= 8 \int_{\Sigma} |\nu_{\bar1\bar 1}|^2 \, dA.$$
\end{proof}
\begin{remark}\label{rk-L}
    It follows from \eqref{eq-Ht} that the area Jacobi operator is of the following form 
    $$\mathcal{L}(V)=2\left.\frac{\partial \vec{H}}{\partial t} \right|_{t=0}=4\bar\nu_{1\bar1}n+ 4\nu_{\bar11}\bar n.$$
\end{remark}

\section{The proof of Theorem 1}\label{sec4} 
To prove Theorem 1, we first relate the second variation of $\mathcal{W}^+$ to the area Jacobi operator $\mathcal{L}$. 

Applying integration by parts to the second variation of $\mathcal{W}^+$ in \eqref{eq-Will} yields  
\begin{equation}\label{eq-W+2}
    \left.\frac{\partial^2}{\partial t^2} \right|_{t=0}\mathcal{W}^+(x_t)
= 8 \int_{\Sigma} \nu_{\bar1\bar 1}\bar \nu_{11}\, dA=-8 \int_{\Sigma} \nu_{\bar1\bar 1 1}\bar \nu_{1}\, dA.
\end{equation}
where $\nu_{\bar1\bar 1 1}$ denotes a third-order  covariant derivative of $\nu$, defined by 
\begin{equation}\label{eq-vb1b1}
    \nu_{\bar1\bar 1 1}\theta+\nu_{\bar1\bar 1 \bar1}\bar\theta\triangleq d\nu_{\bar1\bar1}+(2i\theta_{12}+i\theta_{34})\nu_{\bar1\bar1}.
\end{equation}
We also require another third-order covariant derivative of $\nu$,  defined by   
\begin{equation}\label{eq-vb11}\nu_{\bar1 1\bar 1}\theta+\nu_{\bar1 1 \bar1}\bar\theta\triangleq d\nu_{\bar1 1}+i\theta_{34}\nu_{\bar11},
\end{equation}
together with the following Ricci identity. 
\begin{lemma}\label{lem-Ricci}
\[
\nu_{\bar1\bar{1}1} - \nu_{\bar1 1\bar{1}} = -\nu_{\bar1} \operatorname{Ric}(\bar e_1, \bar e_1).
\]
\end{lemma}
\begin{proof}
By the definition \eqref{eq-vb1b1}, 
\begin{align}\label{eq-vric}
\nu_{\bar1\bar{1}1}{\theta} \wedge \bar\theta
= \bigl( d\nu_{\bar1\bar1} + (2i\theta_{12}+ i\theta_{34}) \nu_{\bar1\bar 1} \bigr) \wedge \bar\theta = d(\nu_{\bar1\bar1}\bar\theta) +  (i\theta_{12}+ i\theta_{34}) \nu_{\bar1\bar 1} \wedge \bar\theta. 
\end{align}
Using \eqref{eq-v1vb1}, we have 
$$\begin{aligned}d(\nu_{\bar1\bar1} \bar\theta)&=d\bigl((i\theta_{12}+ i\theta_{34}) \nu_{\bar 1} -\nu_{\bar11}\theta\bigr)\\
&=-(i\theta_{12}+ i\theta_{34})\wedge d \nu_{\bar 1}-d(x^*\omega_{0\bar0}+x^*\omega_{1\bar1})\nu_{\bar 1}-(d\nu_{\bar11}-i\theta_{12}\nu_{\bar11})\wedge \theta\\
&=-(i\theta_{12}+ i\theta_{34})\wedge (\nu_{\bar 1\bar1}\bar\theta+\nu_{\bar 11}\theta)-x^*(d\omega_{0\bar0}+d \omega_{1\bar1})\nu_{\bar 1}-(d\nu_{\bar11}-i\theta_{12}\nu_{\bar11})\wedge \theta\\
&=-\nu_{\bar 1\bar1}(i\theta_{12}+ i\theta_{34})\wedge \bar\theta-x^*(\Omega_{0\bar0}+ \Omega_{1\bar1})\nu_{\bar 1}-(d\nu_{\bar11}+i\theta_{34}\nu_{\bar11})\wedge \theta.
\end{aligned}$$
Substituting this into \eqref{eq-vric} yields 
$$\nu_{\bar1\bar{1}1}{\theta} \wedge \bar\theta=\nu_{\bar1 1 \bar1}\theta\wedge\bar\theta-x^*(\Omega_{0\bar0}+ \Omega_{1\bar1})\nu_{\bar 1}.$$
The conclusion then follows from \eqref{eq-Ric}. 
\end{proof}

\begin{proof}[\textbf{\textit{Proof of Theorem~\ref{thm-main11}}}] 
Substituting the Ricci identity established in Lemma~\ref{lem-Ricci} into the formula \eqref{eq-W+2}, we obtain 
\begin{equation}\label{eq-W+3}
    \left.\frac{\partial^2}{\partial t^2} \right|_{t=0}\mathcal{W}^+(x_t)=
8 \int_{\Sigma} \Bigl( |\nu_{\bar11}|^2-|\nu_{\bar1}|^2\mathrm{Ric}(e_1,\bar e_1)\Bigr) dA,
\end{equation}
where integration by parts is used again.

Let \(V = \nu e_{0} + \bar\nu \bar e_{0}\) be an eigensection of \(N_\Sigma\) corresponding to the first eigenvalue $\Lambda_1$ of $\mathcal{L}$, i.e., 
\[
\mathcal{L}(V) = -\Lambda_{1}V.
\]
From Remark~\ref{rk-L}, we have 
\[
4\bar\nu_{1\bar1}\,\bar{e}_0+ 4\nu_{\bar11}\,e_0 =- \Lambda_{1} (\bar\nu\,\bar e_{0} + \nu e_{0}),
\]
which gives 
\begin{equation}
  \label{eq-lambda1}
\nu_{\bar11} = -\frac{\Lambda_{1}}{4} \nu, \qquad
\int_{\Sigma} |\nu_{\bar1}|^{2} dA =-\int_{\Sigma} \nu_{\bar1 1}\bar v dA = \frac{\Lambda_{1}}{4} \int_{\Sigma} |\nu|^{2} dA.  
\end{equation}

Using \eqref{eq-W+3} and \eqref{eq-lambda1}, we obtain that along the variation determined by $V$, 
\begin{equation}\label{eq-varW}
 \frac{\partial^{2}}{\partial t^{2}} \bigg|_{t=0} \mathcal{W}^{+}
\leq 8 \int_{\Sigma} \Bigl( |\nu_{\bar11}|^2-\frac{1}{2}\mathfrak{Ric}\,|\nu_{\bar1}|^2\Bigr) dA 
\le \frac{1}{2} \int_{\Sigma} \bigl( \Lambda_{1}^{2} - 2 \Lambda_{1}\mathfrak{Ric} \bigr) |\nu|^{2} dA,    
\end{equation}
where the first inequality follows from
 $$\operatorname{Ric}(e_{1}, \bar e_{1})=\operatorname{Ric}(\mathrm{Re}(e_{1}), \mathrm{Re}(e_{1}))+\operatorname{Ric}(\mathrm{Im}(e_{1}), \mathrm{Im}(e_{1}))\geq \frac{1}{2}\mathfrak{Ric},$$ 
with $\mathrm{Re}(e_{1})$ and $\mathrm{Im}(e_{1})$ denoting the real and imaginary parts of $e_1$. The fact that any complex curve minimizes $\mathcal{W}^+$ implies the left-hand side of \eqref{eq-varW} is nonnegative. Together with $\Lambda_1>0$, this yields the lower bound  
$$\Lambda_1\geq 2\,\mathfrak{Ric}.$$
\end{proof}
\section{The proof of Theorem 2}\label{sec5}
We denote by ${E}_0^{\mathcal{W}^+}$ (resp. $E_0^{\mathcal{A}}$) the space of Jacobi fields associated with  $\mathcal{W}^+$ (resp. $\mathcal{A}$). It follows from \eqref{eq-Area} and \eqref{eq-Will} that 
\begin{equation}\label{eq-Jaco}
    E_0^{\mathcal{W}^+}=\{V=\nu  e_0+\bar\nu\bar e_0 \mid \nu_{\bar1\bar1}=0\},~~~~~~E_0^{\mathcal{A}}=\{V=\nu e_0+\bar\nu\bar e_0\mid  \nu_{\bar1}=0\}.
\end{equation}
Obviously, there holds 
$$E_0^{\mathcal{A}}\subset E_0^{\mathcal{W}^+}.$$
\begin{proof}[\textbf{\textit{Proof of Theorem~\ref{thm-main2}}}] 
Since $\mathrm{Ric}=\mathfrak{c}g$, the first conclusion is a direct corollary of Theorem~\ref{thm-main11}. Meanwhile,  
Remark~\ref{rk-L} yields  
$$|\mathcal{L}(V)|^2=16|\nu_{\bar1 1}|^2, ~~~\langle \mathcal{L}(V), V\rangle =(2\bar\nu \nu_{\bar1 1}+2\nu \bar\nu_{1 \bar1}).$$
Together with $\mathrm{Ric}(e_1, \bar e_1)=\frac{\mathfrak{c}}{2}$ and integration by parts, we obtain from \eqref{eq-W+3} that 
\begin{equation}\label{eq-W+5}
  \left. \frac{\partial^2}{\partial t^2} \right|_{t=0} W^+ 
= \frac{1}{2} \int_{\Sigma} \bigl(|\mathcal{L}(V)|^2 + 2\mathfrak{c} \langle \mathcal{L}(V), V \rangle \bigr) dA.   
\end{equation}

{\bf Claim 1.} {\em If $E_0^{\mathcal{W}^+}\setminus E_0^{\mathcal{A}}\neq\emptyset$, then $\Lambda_1=2\mathfrak{c}$ and $E_0^{\mathcal{W}^+}\setminus E_0^{\mathcal{A}}$ is exactly the first eigenspace of $\mathcal{L}$}.  
\vskip 0.2cm
\noindent To prove this claim, first note that  $\mathcal{L}$ is a nonnegative self-adjoint operator. Hence, for each eigenvalue $\Lambda_j$ of $\mathcal{L}$, we can choose a basis $\{\mathrm{V}_{j_1}, \mathrm{V}_{j_2}, \cdots \mathrm{V}_{j_{k_j}}\}$ of the corresponding eigenspace $E_j^{\mathcal{A}}$ such that the collection  
$$\{\mathrm{V}_{j_l}\mid 1\leq l\leq {k_j}, j\in \mathbb{Z}^{\geq 0}\}$$
forms a complete orthonormal system  of \( L^2(\Gamma(N_\Sigma)) \). 
Let $V\in E_0^{\mathcal{W}^+}\setminus E_0^{\mathcal{A}}$ be a Jacobi field of $\mathcal{W}^+$. Write  
\[V = \sum_{j=0}^{+\infty}\sum_{l=1}^{k_j} \varepsilon_{j_l} \mathrm{V}_{j_l}.\]
Then we have  
\[\mathcal{L}(V) = -\sum_{j=1}^{+\infty}\sum_{l=1}^{k_j}  \Lambda_j \varepsilon_{j_l} \mathrm{V}_{j_l}.\]
and then  
\[\int_{\Sigma}|\mathcal{L}(V)|^2 dA = \sum_{j=1}^{+\infty}\sum_{l=1}^{k_j} \varepsilon_{j_l}^2\Lambda_j^2, \quad\quad \int_{\Sigma}\langle \mathcal{L}(V), V \rangle dA= -\sum_{j=1}^{+\infty}\sum_{l=1}^{k_j} \varepsilon_{j_l}^2 \Lambda_j.\]
Substituting these equalities into \eqref{eq-W+5}, we derive  
\[\sum_{j=1}^{+\infty}\sum_{l=1}^{k_j} \varepsilon_{j_l}^2 \Lambda_j(\Lambda_j - 2\mathfrak{c}) = 0.\]
It follows from 
$$\sum_{j=1}^{+\infty}\sum_{l=1}^{k_j}\varepsilon_{j_l}^2\neq0,~~\text{and}~~\Lambda_{j}>\Lambda_1\geq 2\mathfrak{c},~j\geq 2$$
that 
$$\Lambda_1=2c,~~\text{and}~~\varepsilon_{j_l}=0,~j\geq 2,$$
which implies $V$ is an eigensection associated with the first eigenvalue $\Lambda_1=2\mathfrak{c}$. Therefore, we obtain 
$$E_0^{\mathcal{W}^+}\setminus E_0^{\mathcal{A}}\subset E_1^{\mathcal{A}}.$$
When $\Lambda_1=2\mathfrak{c}$, the converse inclusion follows directly from \eqref{eq-W+5}. Thus Claim $1$ is proved. 

Next, to determine when $E_0^{\mathcal{W}^+}\setminus E_0^{\mathcal{A}}$ is nonempty, we compute the  dimensions of $E_0^{\mathcal{W}^+}$ and $E_0^{\mathcal{A}}$. 

It is well known that the normal bundle \(N_{\Sigma}\) of \(x\) admits a natural holomorphic structure $$\bar{\partial}:\Gamma(N_\Sigma)\rightarrow \Gamma(N_\Sigma\otimes \overline{K_\Sigma}),$$
where $\overline{K_\Sigma}$ is the line bundle defined by $(0,1)$-forms on $\Sigma$. Following \cite{MW}, this structure is obtained by identifying \(N_{\Sigma}\) with the line bundle \(N_{\Sigma}^{\mathbb{C}^{(1,0)}}\)corresponding to the \((1,0)\)-part of the complexified normal bundle \(N_{\Sigma}^{\mathbb{C}}\); the \((0,1)\)-part of the covariant differential then induces a holomorphic structure via the Koszul-Malgrange theorem. By our definition \eqref{eq-Dv}, $V=\nu e_0+\bar\nu\bar e_0\in\Gamma(N_\sigma)$ is holomorphic with respect to  $\bar{\partial}$ if and only if $\nu_{\bar1}=0$. Then, it follows from \eqref{eq-Jaco} that $E_0^{\mathcal{A}}$ is exactly the space $H^0(\Sigma, N_\Sigma)$ of holomorphic section of $N_\Sigma$. 

Similarly, the \((0,1)\)-part of the covariant differential induces on the line bundle \(N_\Sigma \otimes \overline{K_\Sigma}\) a holomorphic structure
\[
\bar{\partial}_1 : \Gamma(N_\Sigma \otimes \overline{K_\Sigma}) \rightarrow \Gamma(N_\Sigma \otimes \overline{K_\Sigma}^2),
\]
again via the Koszul--Malgrange theorem. For a section \(V = \nu e_0 + \bar\nu \bar e_0 \in \Gamma(N_\Sigma)\), \eqref{eq-vb11} implies that \(\nu_{\bar1\bar1} = 0\) if and only if
\(
\bar{\partial}_1 \bar{\partial}(V) = 0.
\)
Therefore, by \eqref{eq-Jaco} we obtain 
\begin{equation}\label{eq-E0W}
E_0^{\mathcal{W}^+}=\{V\in \Gamma(N_\Sigma) \mid \bar{\partial}_1 \bar{\partial}(V)=0\}.
\end{equation}

{\bf Claim 2.} {\em As an operator, $\bar{\partial}:\Gamma(N_\Sigma)\rightarrow \Gamma(N_\Sigma\otimes \overline{K_\Sigma})$ is surjective.}
\vskip 0.2cm
To prove this claim, we need only show the vanishing of the Dolbeaut cohomology $H^{0,1}_{\bar\partial}(\Sigma, N_\Sigma)$. By Dolbeaut Theorem and Serre duality, we have 
$$H^{0,1}_{\bar\partial}(\Sigma, N_\Sigma)\cong H^1(\Sigma, N_\Sigma)\cong H^0(\Sigma, N_\Sigma^*\otimes K_\Sigma)^*.$$
From the Gauss equation \eqref{eq-Gauss2} we obtain
$$
\begin{aligned}\deg(N_\Sigma^*\otimes K_\Sigma)&=c_1(N_\Sigma^*\otimes K_\Sigma)=-\big(c_1(N_\Sigma)+c_1(T\Sigma)\big)\\
&=-\frac{1}{2\pi}\int_{\Sigma}(R_{1212}+R_{1234})dA=-\mathfrak{c}\frac{\mathrm{Area}(\Sigma)}{2\pi},
\end{aligned}$$
where $c_1(\cdot)$ denotes the first Chern number. Since $\mathfrak{c}>0$, it follows that $\deg(N_\Sigma^*\otimes K_\Sigma)<0$, which implies 
$$\dim H^{0,1}_{\bar\partial}(\Sigma, N_\Sigma)= \dim H^0(\Sigma, N_\Sigma^*\otimes K_\Sigma)^*=0.$$
This completes the proof of Claim $2$. 

From \eqref{eq-E0W} and Claim $2$, we derive 
$$\dim E_0^{\mathcal{W}^+}=\dim \ker{\bar\partial_1}+\dim \ker{\bar\partial}= \dim H^0(\Sigma, N_\Sigma \otimes \overline{K_\Sigma})+\dim E_0^{\mathcal{A}}.$$
By Riemann-Roch Theorem, we have 
$$\begin{aligned}\dim H^0(\Sigma, N_\Sigma \otimes \overline{K_\Sigma})=&\deg(N_\Sigma \otimes \overline{K_\Sigma})+1-g+\dim H^0(\Sigma, N_\Sigma^* \otimes K_\Sigma^2)\\
=&\,\mathfrak{c}\frac{\mathrm{Area}(\Sigma)}{2\pi}+1-g+\dim H^0(\Sigma, N_\Sigma^* \otimes K_\Sigma^2).
\end{aligned}$$
So when $g\leq 1$, there holds 
$$\dim E_0^{\mathcal{W}^+}-\dim E_0^{\mathcal{A}}=\dim H^0(\Sigma, N_\Sigma \otimes \overline{K_\Sigma})\geq 1.$$
Together with Claim 1, this yields the second conclusion of this theorem. 
\end{proof}
\begin{remark}
    It follows from the proof of Claim 1 that for complex curves in K\"ahler-Einstein surfaces of nonpositive scalar curvature, $${E}^{\mathcal{W}^+}_0={E}^{\mathcal{A}}_0,$$
    i.e., the Jacobi fields of $\mathcal{W}^+$ are exactly the same as those of $\mathcal{A}$.
\end{remark}

\textbf{Acknowledgement:}
This work was supported by NSFC No. 12171473, as well as the Fundamental Research Funds for the Central Universities.

\noindent{Zhenxiao Xie}\\
\noindent{\small \em School of Mathematical Sciences, Beihang University, Beijing 102206, China.}\\
\noindent{\em Email: xiezhenxiao@buaa.edu.cn}
\\

\end{document}